\documentclass[10pt,reqno]{amsart}

\usepackage{amsmath,amsfonts,amsbsy,amsgen,amscd,mathrsfs,amssymb,amsthm}
\usepackage{enumerate,mathtools}

\usepackage[usenames,dvipsnames]{xcolor}
\usepackage[colorlinks=true,citecolor=blue,linkcolor=blue]{hyperref}

\usepackage{tikz}
\usetikzlibrary{shadings}

\newtheorem{thm}{Theorem}[section]
\newtheorem{lem}[thm]{Lemma}

\newtheorem{rem}[thm]{Remark}

\theoremstyle{definition}

\numberwithin{equation}{section} 
\numberwithin{figure}{section}
\numberwithin{table}{section}

\newcommand{\bE}{\mathbf{E}}
\newcommand{\eps}{\varepsilon}
\newcommand{\bR}{\mathbf{R}}

\newcommand{\Om}{\Omega}

\newcommand{\al}{\alpha}

\begin{document}

\title{Banach space valued Pisier and Riesz type inequalities on discrete cube}

\author{Paata Ivanisvili}
\address{(P.I.) Department of Mathematics, University of California, 
Irvine, CA 92617, and Dept. of Math. North Carolina State University, Raleigh, USA}
\email{pivanisv@uci.edu}


\author{Alexander Volberg}
\address{(A.V.) Department of Mathematics, MSU, 
East Lansing, MI 48823, USA and Hausdorff Center for Mathematics, Bonn, Germany}
\email{volberg@math.msu.edu}

\begin{abstract}
This is an attempt to build Banach space valued theory for certain singular integrals on Hamming cube.
Of course all estimates below are dimension independent, and we tried to find ultimate assumptions on the Banach space for a corresponding operators to be bounded. In certain cases we succeeded, although there are  still many open questions, some of them are listed in the last Section.
Using the approach of \cite{IVHV} and  also quantum random variables approach of  \cite{ELP} we generalize several theorems  of Pisier \cite{P} and Hyt\"onen-Naor \cite{HN}. 
\end{abstract}

\subjclass[2010]{46B09; 46B07; 60E15}

\keywords{Rademacher type; Enflo type;
Pisier's inequality; Banach space theory}

\maketitle

\tableofcontents

\thispagestyle{empty}

\section{Introduction}
\label{intro}

Functions on discrete  cube (Hamming cube) $\{-1,1\}^n$ can be viewed as polynomials of $n$ variables $\eps=(\eps_1, \dots, \eps_n)$, where each $\eps_i$ assumes only two values $\pm 1$. Obviously all such functions are polynomials and the degree is bounded by $n$. Moreover, these are only multi-linear polynomials.

Operator $\partial_i$ works exactly as the usual partial derivative with respect to $i$-th variable.
Another useful operator is 
$$
D_i f:= \eps_i\partial_i f\,.
$$
Obviously $D_i^2=D_i$.

The Laplacian on the discrete cube is defined by
$$
	\Delta f := -\sum_{j=1}^n D_jf.
$$
We denote by $P_t$ the standard heat semigroup on the cube, that is,
$$
	P_t := e^{t\Delta}.
$$
Recall that
$\Delta$ is self-adjoint on $L^2(\{-1,1\}^n)$ with quadratic form
$$
	-\mathbf{E}[f(\varepsilon)\,\Delta g(\varepsilon)] =
	\sum_{j=1}^n \mathbf{E}[D_jf(\varepsilon)\, D_jg(\varepsilon)].
$$

\section{Formula}
\label{f}

The following random variables will appear frequently in the sequel, so we 
fix them once and for all. Given $t>0$, we let 
$\xi(t)$ be a random vector in the cube, independent of $\varepsilon$, 
whose coordinates $\xi_i(t)$ are independent and identically distributed 
with
$$
	\mathbf{P}\{\xi_i(t)=1\}=\frac{1+e^{-t}}{2},\qquad
	\mathbf{P}\{\xi_i(t)=-1\}=\frac{1-e^{-t}}{2}.
$$
We also define the standardized vector $\delta(t)=(\delta_1(t),\dots, \delta_n(t))$ by
$$
	\delta_i(t) := \frac{\xi_i(t)-\mathbf{E}\xi_i(t)}{
	\sqrt{\mathop{\mathrm{Var}}\xi_i(t)}} =
	\frac{\xi_i(t) - e^{-t}}{\sqrt{1-e^{-2t}}}.
$$

The basis for the proof of most results below are the following two
probabilistic representation of the heat semigroup and its discrete
partial derivatives, bot proved in \cite{IVHV}.

\begin{thm}
\label{main-th}
We have
$$
	P_tf(x) = \mathbf{E}[f(x_1\xi_1(t),\ldots,x_n\xi_n(t))]
	\quad  \mbox{for }t\ge 0,
$$
and
\begin{equation}
\label{main-f}
e^{t\Delta} D_j f (\eps)= \frac{e^{-t}}{\sqrt{1-e^{-2t}}} \bE_\xi \Big[\frac{\xi_j(t) -e^{-t}}{\sqrt{1-e^{-2t}}} f(\eps \xi(t))\Big]\,,
\end{equation}
where $\xi_i(t)$ are as above.
\end{thm}

 The formulae can be checked by direct calculation, see also Lemma 2.1 of \cite{IVHV}. The proof is easy, but it turns out to have many consequences, including the solution of Enflo's problem in \cite{IVHV}.

\section{$X$-valued Riesz type estimates from below and from above}
\label{beab}
In what follows we will be interested in the following two different but related inequalities. They hide several parameters. The first one is positive number $a$ that one need to make as small as possible, there is also parameter $p\in [1, \infty)$ (often we had to exclude $p=1$, but not always), they play with each other. But there is a third parameter, namely, a Banach space $X$, this is where the functions on discrete cube assume their values.

Here is the first class of inequalities: for which $(a, p, X)$ this holds:
\begin{equation}
\label{Rabove}
\|\sum_{i=1}^n \delta_i \Delta^{-a} D_i f\|_{L^{p'}(X^*)} \le C\|f\|_{L^{p'}(X^*)}\,?
\end{equation}
It is useful to write a dual version. Let $F(\delta, \eps)$ is an $X$ valued function on $\{-1, 1\}^n\times \{-1, 1\}^n$.
Let $F_i:=\bE_\delta [\delta_i F]$. The dual version of \eqref{Rabove}
is
\begin{equation}
\label{RaboveD}
\|\sum_{i=1}^n  \Delta^{-a} D_i F_i\|_{L^{p}(X)} \le C\|F\|_{L^p(X)}\,?
\end{equation}
We will call this type of inequalities Riesz type estimates from above.  

\medskip

If as $F$ we take a particular sort of functions: $F(\delta, \eps) =\sum_{i=1}^n \delta_i f_i(\eps)$ we get a particular case of the above (so an easier type of estimate, in principle):

\begin{equation}
\label{RbnoD}
\|\sum_{i=1}^n \Delta^{-a} D_i f_i\|_{L^{p}(X)} \le C \|\sum_{i=1}^n \delta_i f_i\|_{L^p(X)}\,.
\end{equation}

It turns out that the following variant of \eqref{RbnoD} is of more interest:

\begin{equation}
\label{Rbelow}
\|\sum_{i=1}^n \Delta^{-a} D_i f_i\|_{L^{p}(X)} \le C \|\sum_{i=1}^n \delta_i D_if_i\|_{L^p(X)}\,.
\end{equation}
\begin{rem}
\label{odd}
It is exactly as  \eqref{RbnoD} if we restrict  \eqref{RbnoD} to functions $f_i$ odd in $\eps_i$. 
\end{rem}
We call  this type of inequalities Riesz type estimates from below.

Consider special case $a=1/2$ in \eqref{Rabove} and the same $a$ and $f-1=\dots-f_n=f$ in \eqref{Rbelow}. Riesz transforms are operators $R_i =\Delta^{-1/2} D_i$, the total Riesz transform is $R:= (R_1,\dots, R_n)$. For example, \eqref{Rbelow} becomes (after using that $\sum D_i =-\Delta$)
$$
\|\Delta^{1/2} f\|_{L^p(X)} \le C_p \|\sum_i \delta_i D_i f\|_{L^p(X)}, 
$$
which is the same as
\begin{equation}
\label{Rbe}
\| f\|_{L^p(X)} \le C_p \|\sum_i \delta_i R_i f\|_{L^p(X)}\,.
\end{equation}
This explains why we used the expression ``estimate of Riesz transform from below''.
Inequality \eqref{Rabove} with $a=1/2$ assumes the form  (we change $X^*$ to $X$ and $p'$ to $p$)
\begin{equation}
\label{Rab}
 \|\sum_i \delta_i R_i f\|_{L^p(X)} \le C_p \|f\|_{L^p(X)} \,.
 \end{equation}
This explains why we used the expression ``estimate of Riesz transform from above''.

\bigskip

\subsection{Comparison of \eqref{RbnoD} and \eqref{Rbelow}}
\label{noD}

Let us consider \eqref{RbnoD} with $a=1/2$ and for scalar valued functions $f$, namely,
\begin{equation}
\label{RbnoD12}
\|\sum_{i=1}^n \Delta^{-a} D_i f_i\|_{p} \le C\Big(\bE_\delta \|\sum_{i=1}^n \delta_i f_i\|_{p}^p\Big)^{1/p}\,.
\end{equation}
If this were true for $p\ge 2$, then by duality, we would prove
\begin{equation}
\label{Lamb}
\||\nabla g|_{\ell^2}\|_{L^{p'}} \lesssim \|\Delta^{1/2} g\|_{L^{p'}}, \quad 1< p'<2,
\end{equation}
but it is know that that inequality is false, it is called Lamberton's counterexample, see \cite{ELP}.

Indeed, let us derive \eqref{Lamb} from \eqref{RbnoD12}. Let $\vec f =(f_1,\dots, f_n)$ be a test vector function, and  let \eqref{RbnoD12} is valid. Then integrating by parts, we het
$$
\langle \nabla g, \vec f\rangle = \bE( \sum_i  D_i f_i) g=  \bE( \sum_i  \Delta^{-1/2}D_i f_i) \Delta^{1/2}g \le
$$
$$
\|\Delta^{1/2} g\|_{p'} \| \sum_i  \Delta^{-1/2}D_i f_i\|_p  \le^{\eqref{RbnoD12}}
$$
$$
C\|\Delta^{1/2} g\|_{p'} \Big(\bE_\delta \|\sum_{i=1}^n \delta_i f_i\|_{p}^p\Big)^{1/p} \le C\|\Delta^{1/2} g\|_{p'} \| |\vec f|_{\ell^2}\|_p,
$$
where the last inequality is Khintchine inequality for scalar functions. Now, in the left hand side, we take the supremum 
over all test functions $\vec f$ such that $\| |\vec f|_{\ell^2}\|_p\le 1$. We get then \eqref{Lamb}, which we know cannot be true. This disproves \eqref{RbnoD12} even for scalar valued functions and $p>2$.

\medskip

\begin{rem}
\label{belowWide}
On the other hand we will see below that  \eqref{Rbelow} with $a=1/2$  is valid in variety of regimes including Banach space valued regimes.  In particular, for $f_1=\dots f_n=f$, $X=\bR$, it is proved in Theorem 1.1 4) of \cite{ELP} by use of quantum random variables. We generalize a little bit the result of \cite{ELP} in Section \ref{qua}.
\end{rem}

\begin{rem}
\label{Bellman}
In Section \ref{qua} we borrow the quantum variable method of \cite{ELP}. But for scalar valued functions we could have done this using the Bellman function method. This will be pursued in a separate note.
\end{rem}

\bigskip

\subsection{Plan}
Most of time we will be dealing with estimates from below. As we already remarked they are somewhat easier than the estimates from above. In particular the stock of Banach spaces will be larger. But let us recall that this stock can be dependent of $a$ and $p$. There are still many open questions concerning this dependence. See below.

\section{Simple corollaries from the formula, $a=1+\eps -\frac1{\max(p, q)}$}
\label{maxpq}

\subsection{A variant of Pisier's inequality in spaces of finite co-type}
\label{sec:co-type}

\begin{thm}
\label{simple1}
Let $X$ have finite  co-type $q, 2\le q<\infty$. Let $\{f_i\}_{i=1}^n$ be $X$-valued function on discrete cube $\{-1, 1\}^n$.
Let $1\le p<\infty$
Then
$$
\big(\bE \|\sum \Delta^{\frac1{\max(p, q)} -1-\eps} D_i f_i\|_X^p\big)^{1/p} \le C(q, p, \eps) \Big(  \bE\Big\|\sum_{j=1}^n \delta_j  f_j(\eps )\Big\|_X^p\Big)^{1/p}\,
$$
where $\{\delta_j\}_{j=1}^n$ are standard Rademacher r.v. independent of $\eps$.
\end{thm}

\begin{proof}
Let us use formula \eqref{main-f} as following: we multiply it by $t^{-\rho}$, $0\le \rho<1$, and integrate $\int_0^\infty\dots dt$. We use an obvious relationship
\begin{equation}
\label{fracPower}
\int_0^\infty \frac1{t^\rho} e^{\Delta t} dt = \Delta^{-1+\rho}, \quad 0\le \rho <1\,.
\end{equation}
Then we get
\begin{equation}
\label{sum1}
\sum_{i=1}^n\Delta^{-1+\rho} D_i f_i = \int_0^\infty \frac{e^{-t}}{t^\rho \sqrt{1-e^{-2t}}}\bE_\xi \Big[\sum_{j=1}^n\delta_j(t)f_j(\eps \xi(t))\Big], \quad \delta_j(t) := \frac{\xi_j(t) -e^{-t}}{\sqrt{1-e^{-2t}}}\,.
\end{equation} 
$$
\Delta^{-1+\rho} \sum_{j=1}^n D_j f_j (\eps)= \int_0^\infty\frac{e^{-t}}{t^\rho\sqrt{1-e^{-2t}}} \bE_\xi \Big[\sum_{j=1}^n \delta_j(t) [ f_j](\eps \xi(t))\Big]dt\,,
$$
We put norm on the both parts of that equality.
$$
(\bE_\eps \|\sum_{j=1}^n \Delta^{-1+\rho} D_j f_j (\eps)\|^p)^{1/p} \le \int_0^\infty \Big(\bE_\xi\bE_\eps \Big\|\sum_{j=1}^n \delta_j(t) [f_j](\eps \xi(t))\Big\|^p\Big)^{1/p} \, \frac{e^{-t}}{t^{\rho}\sqrt{1-e^{-2t}}} dt=
$$
$$
 \int_0^\infty \Big(\bE_\xi\bE_\eps \Big\|\sum_{j=1}^n \delta_j(t) [f_j](\eps )\Big\|^p\Big)^{1/p} \, \frac{e^{-t}}{t^{\rho}\sqrt{1-e^{-2t}}} dt,
$$
where we used that for every fixed $t, \xi(t)$ the distribution of $\eps\to [\Delta f_j](\eps\xi (t))$ is the same as that of $\eps\to [\Delta f_j](\eps )$. 

We continue by introducing symmetrization by means of $\{\xi_i'(t)\}$ independent from all $\{\xi_j\}_{j=1}^n$ and having the same distribution as $\xi_i$, $i=1, \dots, n$, exactly as Theorem 4.1 of \cite{IVHV} and use contraction principle  of Maurey and Pisier  \cite{P}, Proposition 3.2,  for $q$ co-type valued $f_i$:

$$
 \int_0^\infty \Big(\bE_\eps \bE_{\xi, \xi'} \Big\|\sum_{j=1}^n  \frac{\xi_j(t) -\xi'_j(t)}{\sqrt{1-e^{-2t}}} [ f_j](\eps )\Big\|^p\Big)^{1/p} \, \frac{e^{-t}}{t^{\rho}\sqrt{1-e^{-2t}}} dt\le
$$
$$
C(q, p)  \int_0^\infty \Big(\bE_\eps \bE_\delta \Big\|\sum_{j=1}^n \delta_j [f_j](\eps )\Big\|^p\Big)^{1/p} \, \frac{e^{-t}}{t^\rho(1-e^{-2t})^{1-\frac{1}{\max (q, p)} }}dt\,,
$$
where we used Theorem 4.1 from \cite{IVHV} and the fact that the co-type of $X$ is $q<\infty$. 
The convergence of integral is ensured by the choice
$$
\rho= \frac1{\max(p,q)} -\eps, \quad \forall \eps>0.
$$

\end{proof}

\begin{rem}
\label{const1}
One can see that $C(q, p) = C_q(X) \max(1, \sqrt{q/p})$. If $\rho=0$, then we get the following constant growth for large $p\ge q$:
\begin{equation}
\label{const2}
\big(\bE \|\sum \Delta^{ -1} D_i f_i\|_X^p\big)^{1/p} \le C_q(X) \, p\, \Big(  \bE\Big\|\sum_{j=1}^n \delta_j  f_j(\eps )\Big\|^p\Big)^{1/p}\,, \quad q\le p\,.
\end{equation}

Using the fact that $D_i^2 =D_i$ we get the following less general (but as we will see later ``more correct'') version
\begin{equation}
\label{constD2}
\big(\bE \|\sum \Delta^{ -1} D_i f_i\|_X^p\big)^{1/p} \le C_q(X) \, p\, \Big(  \bE\Big\|\sum_{j=1}^n \delta_j  D_jf_j(\eps )\Big\|^p\Big)^{1/p}\,, \quad q\le p\,.
\end{equation}

Choosing $f_i = D_i (f-\bE f)$ and using that $\sum_i D_i = -\Delta$ we get back Pisier inequality for function with values in a finite co-type (co-type $q$) Banach space:
\begin{equation}
\label{const2f}
\big(\bE \| f-\bE f\|_X^p\big)^{1/p} \le C_q(X) \, p\, \Big(  \bE\Big\|\sum_{j=1}^n \delta_j  D_jf(\eps )\Big\|^p\Big)^{1/p}\,, \quad q\le p\,.
\end{equation}
\end{rem}

\begin{rem}
\label{sung1}
We call the attention of the reader to the ``wrong'' constant growth with respect to $p$. One would expect to have for large $p$ the estimate of type $C(q, X) \sqrt{p}$. We actually know that for scalar valued $f$ (that is for $X=\bR$) the constant does grow at most like $C\sqrt{p}$. However, our formula did not give us that. By more efforts one can get the right estimate $C(q, X) \sqrt{p}$ at least for $X= L^q, 1\le q\le 2$, $p$ large. Hence, in particular one can restore the right growth for scalar valued functions. It would be unfortunate if one were unable to do that as the right growth of constant with respect to $p$ is equivalent to subgaussian concentration for Lipschitz functions on discrete cube.
\end{rem}


\begin{thm}
\label{co-type}
For any functions $f_i: \{-1,1\}^n \to X$, $i=1, \dots, n$, $p\in [1, \infty)$, we have
\begin{equation}
\label{Deltafi}
\bE\|\sum_{i=1}^n \Delta^{-a}D_i f_i\|_X^p \le  C \bE\|\sum_{i=1}^n \delta_i f_i\|_X^p
\end{equation}
for all $a$ such that $1-\frac1{\max(p,q)}<a \le 1$ if and only if $(X,\|\cdot\|)$ be a Banach space of {\it finite co-type} $q$. Here $\delta_i$ are standard Rademacher random variables independent of r. v.  $\{\eps\}_{i=1}^n$. If this is true for $a=1$, it is true for all $a$ such that $1-\frac1{\max(p,q)}<a \le 1$.
\end{thm}

\begin{rem}
In \cite{HN} this inequality for $a=1$ and was proved  for $UMD^+$ Banach spaces.
\end{rem}

\begin{proof}
The proof of \eqref{Deltafi} for all finite co-type Banach spaces  $X$ follows immediately from Theorem \ref{simple1}.

If one chooses $a=1$ and $f_i= \Delta^{-1} D_i (f-\bE f)$, then one restores  Proposition 4.2 of \cite{IVHV}. In other words, this becomes Pisier inequality \cite{P} with constant independent of $n$ valid for all Banach spaces of finite co-type. We know by Talagrand's example that Pisier inequality can hold {\it only} for $X$ of finite co-type. So if the above inequality holds, then $X$ has to be of finite co-type and then this inequality will hold for all $a$ such that $1-\frac1{\max(p,q)}<a \le 1$ by Theorem \ref{simple1}.
\end{proof}


\subsection{Scalar functions, $1\le p\le2$}
\label{1p2}

$$
\bE_\eps \Big| \bE_\xi \sum_i \delta_i(t)  D_i f_i (\eps\xi(t))\Big|^p  \le \bE_\eps  \bE_\xi\Big| \sum_i \delta_i(t)  D_i f_i (\eps\xi(t))  \Big|^p  = \bE_\xi\bE_\eps\Big| \sum_i \delta_i(t)  D_i f_i (\eps\xi(t))  \Big|^p   =
$$
$$
 \bE_\xi\bE_\eps\Big| \sum_i \delta_i(t)  D_i f_i (\eps)  \Big|^p   =\bE_\eps\bE_\xi\Big| \sum_i \delta_i(t)  D_i f_i (\eps)  \Big|^p  \le^{Holder, \, p\le 2}
 $$
 $$
 \bE_\eps\Big(\bE_\xi\Big| \sum_i \delta_i(t)  D_i f_i (\eps)  \Big|^2  \Big)^{p/2}= \bE_\eps \Big(\sum_i (D_if_i)^2\Big)^{p/2}\asymp
 $$
 $$
 \bE_\eps \bE_\delta \Big| \sum_i \delta_i D_i f_i\Big|^p\,.
 $$
 Thus we just proved (trivially as the reader just saw) inequality \eqref{Exi} for scalar functions and in the regime $1\le p\le 2$. This proves the analog of Theorem \ref{DPt}.
 \begin{thm}
 \label{1p2-th}
 Let $1\le p\le2$, then for scalar valued function we have
 \begin{equation}
\label{Ptp2}
\bE_\eps \|\sum_i D_i P_t f_i\|^p \le C_p\, \Big(\frac{1}{\sqrt{e^{2t}-1}}\Big)^{p}\bE_{\delta, \eps} \|\sum_i \delta_i D_i f_i\|^p\,.
\end{equation}
\end{thm}

\subsection{Regime $1\le q, p\le 2$}
\label{1qp2}

We follow exactly the same steps as in Subsection \ref{1p2}, but instead of absolute value we have now $L^q(\nu)$-norm, that we denote by $\Big|\cdot\Big|_q$.

$$
\bE_\eps \Big| \bE_\xi \sum_i \delta_i(t)  D_i f_i (\eps\xi(t))\Big|_q^p  \le \bE_\eps  \bE_\xi\Big| \sum_i \delta_i(t)  D_i f_i (\eps\xi(t))  \Big|_q^p  = \bE_\xi\bE_\eps\Big| \sum_i \delta_i(t)  D_i f_i (\eps\xi(t))  \Big|_q^p   =
$$
$$
 \bE_\xi\bE_\eps\Big| \sum_i \delta_i(t)  D_i f_i (\eps)  \Big|_q^p   =\bE_\eps\bE_\xi\Big| \sum_i \delta_i(t)  D_i f_i (\eps)  \Big|_q^p  \le^{Holder, \, p\le 2}
 $$
 $$
 \bE_\eps\Big(\bE_\xi\Big| \sum_i \delta_i(t)  D_i f_i (\eps)  \Big|_q^2  \Big)^{p/2}\,.
  $$
  At this place in Subsection \ref{1p2} we used the orthogonality of $\{\delta_i(t)\}$ with respect to $\bE_\xi$.
  Now we cannot do that because we have $\Big|\cdot \Big|_q^2$ and not just the simple absolute value squared $\Big|\cdot \Big|^2$. 
  However,
  we have shown that
  \begin{equation}
  \label{however}
\bE_\eps \Big| \bE_\xi \sum_i \delta_i(t)  D_i f_i (\eps\xi(t))\Big|_q^p  \le  \bE_\eps\bE_\xi\Big| \sum_i \delta_i(t)  D_i f_i (\eps)  \Big|_q^p \,.
\end{equation}
  
  \bigskip
  
 We use the following contraction principle is a classical result of Maurey and 
Pisier (see, e.g., \cite{P}, Proposition 3.2). We spell out a version 
with explicit constants following \cite{IVHV}.

\begin{thm}
\label{thm:mp}
Let $(X,\|\cdot\|)$ be a Banach space of co-type $q<\infty$,
let $\eta_1,\ldots,\eta_n$ be i.i.d.\ symmetric random variables,
and let $\delta$ be uniformly distributed on $\{-1,1\}^n$. Then
for any $n\ge 1$, $x_1,\ldots,x_n\in X$, and $1\le p<\infty$, we have
$$
	\Bigg(
	\mathbf{E}\Bigg\|\sum_{j=1}^n \eta_j x_j\Bigg\|^p\Bigg)^{1/p}
	\le
	L_{q,p}
	\int_0^\infty \mathbf{P}\{|\eta_1|>t\}^{\frac{1}{\max(q,p)}}dt\,
	\Bigg(
	\mathbf{E}\Bigg\|\sum_{j=1}^n \delta_j x_j\Bigg\|^p
	\Bigg)^{1/p}
$$
with $L_{q,p}=L\,C_q(X)\max(1,(q/p)^{1/2})$,
where $L$ is a universal constant.
\end{thm}

Let $\xi'(t)$ be an independent copy of $\xi(t)$.
We first note that 
$$
	\int_0^\infty \mathbf{P}\{|\xi_j(t)-\xi_j'(t)|>s\}^{1/r}ds
	=
	2^{1-1/r} (1-e^{-2t})^{1/r}.
$$
Our Banach space $X=L^q(\nu)$, $1\le q\le 2$ so the co-type is $2$.
We will use this now with $r=\max(p, 2)$.
Then the right hand side of \eqref{however} can be estimated as follows:
\begin{align*}
        \bE_\eps\bE_\xi
        \Big|\sum_{j=1}^n
        \delta_j(t) D_jf(\varepsilon)\Big|_q^p
	&\le
	\frac{1}{\sqrt{1-e^{-2t}}^{p}}
        \bE_\eps\bE_\xi
        \Big|\sum_{j=1}^n
        (\xi_j(t)-\xi_j'(t)) D_jf(\varepsilon)\Big|_q^p
	\\
	&\le
	2L_{q,p} 
	(1-e^{-2t})^{\frac{p}{\max(2,p)}-\frac{p}{2}}
        \bE_\eps\bE_\delta
        \Big|\sum_{j=1}^n
        \delta_j D_jf(\varepsilon)\Big|_q^p,
\end{align*}
where we used Jensen's inequality in the first line and we applied
Theorem \ref{thm:mp} in the second line. Notice that the power $\frac{p}{\max(2,p)}-\frac{p}{2}=0$ of course as $p\le 2$ now.

Combining the latter inequality and \eqref{however} we conclude that inequality \eqref{Exi} is proved also in the regime $X=L^q(\nu)$, $1\le q, p\le 2$.


\section{Pisier type inequalities with $X$-valued functions and ultimate assumptions on $X$}
\label{Pisier}

In this Section we are going to consider results of the following nature:
\begin{thm}
\label{co-type1}
For any functions $f_i: \{-1,1\}^n \to X$, $i=1, \dots, n$, $p\in [1, \infty)$, we have
\begin{equation}
\label{Deltafi}
\bE\|\sum_{i=1}^n \Delta^{-1}D_i f_i\|_X^p \le  C(q, p) \bE\|\sum_{i=1}^n \delta_i  f_i\|_X^p
\end{equation}
if and only if $(X,\|\cdot\|)$ be a Banach space of {\it finite co-type} $q$.
\end{thm}

It is very well known that for scalar  functions with zero average and any $1<p<\infty$, $\al>0$,
\begin{equation}
\label{Da}
\|f\|_p\le C\|\Delta^\al f\|\,, \text{or}\,\, \|\Delta^{-\al_1}f\|_p\le C\|\Delta^{-\al_2} f\|, \,\, \al_1>\al_2\ge 0\,.
\end{equation} 
This follows from exponential contraction of semigroup $P_t$ ($k>0$ below)
\begin{equation}
\label{cPt}
\|P_t f\|_p \le e^{-kt} \|f\|_p, \quad \forall t>0\,.
\end{equation}

\begin{rem}
\label{Xnontrivtype}
But only slightly worse exponential estimate holds for $X$-valued functions, if $X$ is of non-trivial type.
This is Theorem 5.1 in \cite{MN}. Moreover, such decay of $P_t$ in $L^p$-metric (for any $1<p<\infty$)  when $t\to\infty$ is equivalent to $X$ being of non-trivial type, see Theorem 5.2 \cite{MN}.
 \end{rem}

However, this implication happens only for $X$ of non-trivial type, and the scope of Theorem \ref{co-type1} is all spaces of finite co-type. 
\bigskip

Theorem \ref{co-type1} is proved in \cite{IVHV} but we repeat it here.

\begin{proof}
The proof of \eqref{Deltafi} for all finite co-type Banach spaces  $X$ follows immediately from the formula
\begin{equation}
\label{main}
e^{-t\Delta} D_j f (\eps)= \frac{e^{-t}}{\sqrt{1-e^{-2t}}} \bE_\xi \Big[\frac{\xi_j(t) -e^{-t}}{\sqrt{1-e^{-2t}}} f(\eps \xi(t))\Big]\,,
\end{equation}
where $\xi_i(t)$ assume values $\pm 1$ with probability $\frac12(1\pm e^{-t})$, and are mutually independent and independent from $\eps_j, j=1, \dots, n$. This is formula  from Section \ref{f}, see also Lemma 2.1 of \cite{IVHV}.

Adding and applying $\Delta= \Delta_\eps$, we get (we denote $\delta_j(t):= \frac{\xi_j(t) -e^{-t}}{\sqrt{1-e^{-2t}}}$)

$$
\Delta e^{-t\Delta} \sum_{j=1}^n D_j f_j (\eps)= \frac{e^{-t}}{\sqrt{1-e^{-2t}}} \bE_\xi \Big[\sum_{j=1}^n \delta_j(t) [\Delta f_j](\eps \xi(t))\Big]\,,
$$
After integrating in $t$, we get
$$
-\sum_{j=1}^n D_j f_j (\eps) = \int_0^\infty  \bE_\xi \Big[\sum_{j=1}^n \delta_j(t) [\Delta f_j](\eps \xi(t))\Big]\, \frac{e^{-t}}{\sqrt{1-e^{-2t}}} dt\,.
$$
Hence, introducing symmetrization by means of $\{\xi_i'(t)\}$ independent from all $\{\xi_j\}_{j=1}^n$ and having the same distribution as $\xi_i$, $i=1, \dots, n$, we write
$$
(\bE_\eps \|\sum_{j=1}^n D_j f_j (\eps)\|^p)^{1/p} \le \int_0^\infty \Big(\bE_\xi\bE_\eps \Big\|\sum_{j=1}^n \delta_j(t) [\Delta f_j](\eps \xi(t))\Big\|^p\Big)^{1/p} \, \frac{e^{-t}}{\sqrt{1-e^{-2t}}} dt=
$$
$$
 \int_0^\infty \Big(\bE_\xi\bE_\eps \Big\|\sum_{j=1}^n \delta_j(t) [\Delta f_j](\eps )\Big\|^p\Big)^{1/p} \, \frac{e^{-t}}{\sqrt{1-e^{-2t}}} dt,
$$
where we used that for every fixed $t, \xi(t)$ the distribution of $\eps\to [\Delta f_j](\eps\xi (t))$ is the same as that of $\eps\to [\Delta f_j](\eps )$. We continue:
$$
 \int_0^\infty \Big(\bE_\eps \bE_{\xi, \xi'} \Big\|\sum_{j=1}^n  \frac{\xi_j(t) -\xi'_j(t)}{\sqrt{1-e^{-2t}}} [\Delta f_j](\eps )\Big\|^p\Big)^{1/p} \, \frac{e^{-t}}{\sqrt{1-e^{-2t}}} dt\le
$$
$$
C(q, p)  \int_0^\infty \Big(\bE_\eps \bE_\delta \Big\|\sum_{j=1}^n \delta_j [\Delta f_j](\eps )\Big\|^p\Big)^{1/p} \, \frac{e^{-t}}{(1-e^{-2t})^{1-\frac{1}{\max (q, p)} }}dt\,,
$$
where we used Theorem 4.1 from \cite{IVHV} and the fact that the co-type of $X$ is $q<\infty$. The last expression is bounded by $C(q, p) \max (q, p)\Big( \bE_\eps \bE_\delta \Big\|\sum_{j=1}^n \delta_j [\Delta f_j](\eps )\Big\|^p\Big)^{1/p}$, and we are done.

Notice that, as it follows from \cite{IVHV}, the condition of having finite co-type is not only sufficient for inequality \eqref{Deltafi} to hold, but it is also necessary.

\end{proof}

\subsection{Pisier inequality's constant}
\label{log}

For any Banach space $X$ with no restriction Pisier's inequality claims
\begin{equation}
\label{Pisier}
(\bE\|f-\bE f\|^p)^{1/p} \le  C(n) (\bE\|\sum_{i=1}^n \delta_i \Delta f_i\|^p)^{1/p}\,,
\end{equation}
where
$$
C(n) \le C \log n\,.
$$
In \cite{HN} it was shown that
\begin{equation}
\label{logC}
C(n) \le  \log n + C\,.
\end{equation}
We will show, using the formula from Section \ref{f},  that 
\begin{equation}
\label{lolo}
C(n) \le \log n + \log\log n +C\,,
\end{equation}

\begin{rem}
This is still worse than in  in \cite{HN}, where a very elegant proof of \eqref{logC} is given. So far we cannot get this result of \cite{HN} by our method.
\end{rem}



To show \eqref{lolo} we notice that for any $p$, $\|D_j f\|_p \le \|f\|_p$, and, hence, 
$$
\|\Delta f \|_p \le n\|f\|_p\,.
$$
This gives us immediately and for any $\tau >0$ and  $f: \,\bE f=0$ that $\|P_{-\tau} f\|_p \le e^{\tau n} \|f\|_p$, or
\begin{equation}
\label{compare}
\|f\|_p \le e^{\tau n} \|P_\tau f\|_p\,.
\end{equation}

Notice several things.
Firstly we already proved  for $f$ with $\bE f=0$ that
\begin{equation}
\label{pi}
(\bE\|P_\tau f\|^p)^{1/p} \le \int_\tau^\infty \frac{e^{-t}}{1-e^{-2t}} \Big( \bE_\eps \bE_{\xi, \xi'} \|\sum_{j=1}^n (\xi_j(t)-\xi_j'(t)) D_j f(\eps)\|^p\Big)^{1/p} dt\,,
\end{equation}
where $\xi_j(t)$ were introduced above.
By Proposition 3.2.10 (Kahane contraction principle) of \cite{HVNVW}  we can extend this estimate (everything is real valued):
\begin{equation}
\label{pi1}
(\bE\|P_\tau f\|^p)^{1/p} \le 2\int_\tau^\infty \frac{e^{-t}}{1-e^{-2t}} \Big( \bE_\eps \bE_{\delta} \|\sum_{j=1}^n \delta_j D_j f(\eps)\|^p\Big)^{1/p} dt\,,
\end{equation}
where $\{\delta_j\}_{j=1}^n$ is another set of Rademacher standard variables, independent from $\{\eps_i\}_{i=1}^n$. Obviously, inequality \eqref{pi1} can be rewritten as
\begin{equation}
\label{pi2}
(\bE\|P_\tau f\|^p)^{1/p} \le 2\Big(\frac12\log \frac{1+ e^{-\tau}}{1-e^{-\tau}}\Big)\Big( \bE_\eps \bE_{\delta} \|\sum_{j=1}^n \delta_j D_j f(\eps)\|^p\Big)^{1/p} \,.
\end{equation}

We are left to compare $(\bE\|P_\tau f\|^p)^{1/p}$ and $(\bE\| f\|^p)^{1/p}$ as in \eqref{compare}.

Now we combine \eqref{compare} with \eqref{pi2} to get  
$$
C(n) \le \min_{\tau>0} e^{\tau n} \frac{1+ e^{-\tau}}{1-e^{-\tau}} = \min_{0<r<1} r^{-n} \frac{1+ r}{1-r} = \log n +\log\log n +C\,.
$$
\begin{rem}
Looking at Section 6 in Talagrand's paper \cite{T93} one can notice, that (by twitching the example just a tiny bit), one gets the estimate from below of Pisier's constant:
$$
C(n) \ge (\frac12 - \delta) \log n - C_\delta\,.
$$
\end{rem}

\subsection{Yet another generalization of Pisier's inequality}
\label{F}

Let $F$ be a function of $\{-1, 1\}^n\times \{-1, 1\}^n$ with values in the Banach space $X$, and 
$F_j(\eps) = \bE_\delta \delta_j F(\eps, \delta)$.  For the special case $F(\eps, \delta) =\sum_{j=1}^n F_j(\eps) \delta_j$, inequality
\begin{equation}
\label{F1}
(\bE_\eps\|\sum_{j=1}^n \Delta^{-1} D_j F_j\|^p)^{1/p} \le C(p, n) (\bE_{\delta, \eps} \|F\|^p)^{1/p},\, 1\le p\le  \infty,
\end{equation}
is exactly \eqref{Deltafi}. For such  very special $F=\sum_{j=1}^n F_j(\eps) \delta_j$ we know three things: 1) for general Banach space $X$, $C(p, n) \lesssim \log n$, 2) this is sharp growth, 3) $C(p, n)\le C(p, q)<\infty$ if $X$ is of finite co-type and if $X$ is not of finite co-type, the constant  can grow logarithmically with $n$, \cite{T93}.

\bigskip

However, it is interesting to ask for general function $F(\eps, \delta)$ not just $ =\sum_{j=1}^n F_j(\eps) \delta_j$ what happens with \eqref{F1}, namely, 

A) for what Banach spaces constant does not depend on $n$? 

B) what is the worst growth of constant with $n$ for general Banach space $X$?

\bigskip

\begin{rem}
In \cite{HN} the example of co-type $2$ space is considered, namely, $X= L^1(\{-1, 1\}^n)$, for which constant grows at least as $\sqrt{n}$.  Below we show that this is the worst behavior for an arbitrary Banach space  {\it of finite co-type}. Thus, conceptually, \eqref{F1} turns out to be very different from  \eqref{Deltafi} or from the original Pisier inequality.  
\end{rem}

\bigskip

Before formulating theorem let us consider the dual inequality to \eqref{F1}:
\begin{equation}
\label{dF}
\bE_{\delta, \eps}\|\sum_{j=1}^n \delta_j \Delta^{-1} D_j g(\eps)\|^p \le C(p, n) \bE_\eps\|g\|^p, \, 1\le p\le \infty\,.
\end{equation}
In cases when $C(p, n)<\infty$ independent of $n$ and for $1<p<\infty$, this is one of the  {\it  typical Riesz projection} inequalities.

\begin{thm}
\label{F1thm}
Let $X$ be of finite co-type $q$. Then inequality \eqref{F1} \textup(and thus \eqref{dF}\textup) holds with constant $C(p, q)\sqrt{n}$.
The growth of constant cannot be improved in this class of $X$.
\end{thm}

\begin{proof}
We again use the same formula \eqref{main}, now in the following form:

$$
\Delta e^{-t\Delta} \sum_{j=1}^n \Delta^{-1}D_j F_j (\eps)= \frac{e^{-t}}{\sqrt{1-e^{-2t}}} \bE_\xi \Big[\sum_{j=1}^n \delta_j(t) [ F_j](\eps \xi(t))\Big]\,,
$$
Hence
$$
\bE_{\delta, \eps}\|\sum_{j=1}^n \Delta^{-1} D_j F_j\|^p \lesssim\int_0^\infty\frac{e^{-t}}{\sqrt{1-e^{-2t}}}\bE_\delta\bE_\xi\bE_\eps\|\sum_{j=1}^n \delta_j\delta_j(t) F(\eps\xi, \delta)\|^p dt=
$$
$$
\int_0^\infty\frac{e^{-t}}{\sqrt{1-e^{-2t}}}\bE_\delta\bE_\xi\bE_\eps\|\sum_{j=1}^n \delta_j\delta_j(t) F(\eps, \delta)\|^p dt,
$$
where we used that for every fixed $\xi$ the distribution of $\eps\to  F(\eps\xi, \delta)$ is the same as the distribution of $\eps\to  F(\eps \delta)$.
Using now finite co-type as before, we continue to write
\begin{equation}
\label{co-typeF}
\lesssim \int_0^\infty \frac{e^{-t}}{(1-e^{-2t})^{1- \min(1/p, 1/q)}} \bE_\delta\bE_\eps\bE_{\delta'} \|\sum_{j=1}^n \delta_j\delta_j' F(\eps, \delta)\|^p dt \lesssim
\end{equation}
$$
C(p, q) \bE_\delta\Big[\bE_\eps \|F(\eps, \delta)\|^p  \bE_{\delta'}|\sum_{j=1}^n  \delta_j\delta_j' |^p \Big] \le C'(p, q) n^{p/2}\bE_\eps \bE_\delta \|F(\eps, \delta)\|^p\,.
$$

\end{proof}

\begin{thm}
\label{F1type}
Let $X$ be of non-trivial  type $s\in (1, 2]$. Let $1< p<\infty$. Then inequality \eqref{F1} \textup(and thus \eqref{dF}\textup) holds with constant $C(s, p)<\infty$ independently of $n$.
\end{thm}

\begin{proof}
First we prove inequality for $s\le p<\infty$.
Non-trivial type implies finite co-type $s'<\infty$, and this, by Theorem \ref{co-type},  implies  the following:
\begin{equation}
\label{F2}
\bE_\eps\|\sum_{i=1}^n \Delta^{-1}D_i F_i\|^p \le  C(s', p) \bE_{\delta, \eps}\|\sum_{i=1}^n \delta_i F_i\|^p\,, \, 1\le p <\infty\,.
\end{equation}
Now we use Kahane--Khintchine inequality to write
\begin{equation}
\label{KK}
 \bE_\delta\|\sum_{i=1}^n \delta_i F_i\|^p \le  C(s, p)\Big( \bE_\delta\|\sum_{i=1}^n \delta_i F_i\|^s\Big)^{p/s}\,.
\end{equation}
But the expression $\sum_{i=1}^n \delta_i F_i(\eps_0)$ is the Rademacher projection of function $\delta\to F(\eps_0, \delta)$. As $X$ has a non-trivial type $s$, by Pisier's theorem  (see \cite{PAnnals} and Theorem 7.4.28 of \cite {HVNVW2}) it has Rademacher projection bounded independently of $n$ in $L^s(X)$. So 

$$
\Big(\bE_\delta\|\sum_{i=1}^n \delta_i F_i(\eps_0)\|^s \Big)^{p/s}\le C'(s) \Big(\bE_\delta\|F(\eps_0, \delta)\|^s \Big)^{p/s}\le C'(s) \bE_\delta\|F(\eps_0, \delta)\|^p \,.
$$
Now we combine that inequality with  \eqref{KK} for a fixed $\eps=\eps_0$. We are left to integrate in $\eps_0$ and to use \eqref{F2}.

Type $s$ implies type $s_1\in (1, s]$, therefore, in the above reasoning any $p\in (1, \infty)$ can be used.
\end{proof} 

\begin{rem}
In \cite{HN} inequality \eqref{F1} \textup(and thus \eqref{dF}\textup) was proved for $X$ such that $X^*\in UMD$ \textup(in fact a potentially bigger class $UMD^+$ was involved\textup). As non-trivial-type class is self dual, and also strictly wider than $UMD$, so the latter theorem generalizes Theorem 1.4 of \cite{HN}.
\end{rem}
\begin{rem}
It is interesting to notice that the proof of \cite{HN} is based on a formula that means that operators $\Delta^{-1} D_j$ are 
``averages of martingale transforms". As these operators are ``the second order Riesz transforms" on Hamming cube \textup(in fact, $\Delta^{-1} D_j= \Delta^{-1} D_j^2$\textup), it is natural to compare this averaging of martingale transforms to Riesz transforms with the same idea recently widely used in harmonic analysis, see, e. g. \cite{DV}, \cite{PTV}, \cite{NTV1}, \cite{NTV2}. Paper  \cite{DV} is devoted to representing second order Riesz transforms in euclidean space as averaging of martingale transforms, in \cite{PTV} the similar result is proved for the first order Riesz transforms.
\end{rem}

\begin{thm}
\label{F1log}
Let $X$ be an arbitrary Banach space, and $1\le p<\infty$. Then inequality \eqref{F1} \textup(and thus \eqref{dF}\textup) hold with constant $C(p, n)\lesssim n\log n$.
\end{thm}

\begin{proof}
Let us reason precisely like at the end of Section \ref{log}, but in $\int_\tau^\infty ...dt$ let us use Theorem \ref{F1thm}. In other words: we do not use \eqref{co-typeF} (we cannot), but instead use
$$
(\bE_{\delta, \eps}\|P_\tau\sum_{j=1}^n \Delta^{-1} D_j F_j\|^p)^{1/p} \lesssim\int_\tau^\infty\frac{e^{-t}}{\sqrt{1-e^{-2t}}}(\bE_\delta\bE_\xi\bE_\eps\|\sum_{j=1}^n \delta_j\delta_j(t) F(\eps\xi, \delta)\|^p)^{1/p} dt=
$$
$$
\int_\tau^\infty\frac{e^{-t}}{1-e^{-2t}}(\bE_\delta\bE_\xi\bE_\eps\|\sum_{j=1}^n \delta_j(\xi_j(t)-\xi_j'(t)) F(\eps, \delta)\|^p )^{1/p}dt,
$$
Now we use Kahane contraction principle:
$$
(\bE_{\delta, \eps}\|P_\tau\sum_{j=1}^n \Delta^{-1} D_j F_j\|^p)^{1/p} \lesssim\log\frac{1+e^{-\tau}}{1-e^{-\tau}} (\bE_\delta [\bE_\eps \|F(\eps, \delta)\|^p )^{1/p}\cdot (|\sum_{j=1}^n \delta_j|^p])^{1/p} \le 
$$
$$
n\log\frac{1+e^{-\tau}}{1-e^{-\tau}} (\bE_{\eps, \delta} \|F(\eps, \delta)\|^p)^{1/p} \,.
$$
Now using again that $\|f\|_p \le e^{\tau  n} \|P_\tau f\|_p$, we get
$$
(\bE_{\delta, \eps}\|\sum_{j=1}^n \Delta^{-1} D_j F_j\|^p)^{1/p} \le n\Big(\min_{0<r<1} r^{-pn} \log \frac{1+r}{1-r} \Big)^{1/p}\lesssim n\log n (\bE_{\eps, \delta} \|F(\eps, \delta)\|^p)^{1/p}\,.
$$
\end{proof}

\begin{rem}
This theorem sounds a bit silly. It should be $\lesssim \sqrt{n}$ for all Banach spaces.  It can be that we missed something simple. On the other hand, it may be a worthwhile exercise to ``marry" the example giving $\sqrt{n}$ in \cite{HN} and Talagrand's example from \cite{T93}, to possibly have a Banach space with behavior of constant in \eqref{F1}, which is worse than $\sqrt{n}$. We did not try so far.
\end{rem}

\section{Quantum random variables, $\bE_\eps \|\sum_i D_i \Delta^{-1/2}  f_i\|_X^p $, $X=L^q$}
\label{qua}

In this section we abandon the idea of  Section \ref{f}. Instead we borrow the idea of Francoise Lust-Piquard \cite{FLP}, \cite{ELP}. The idea is to use non-commutative (quantum) random variables to prove a scalar (and vector) valued Riesz transform estimates. 

The idea below is borrowed from  Francoise Lust-Piquard in \cite{FLP}, \cite{ELP}.

We wish to prove the following.

\begin{thm}
\label{Delta12}
Let $X= L^q$, $2\le q\le p<\infty$.   Then
\begin{equation}
\label{Delta-m}
\bE_\eps \|\sum_i D_i \Delta^{-1/2}  f_i\|_X^p \le C(p,q) \bE_\delta\bE_{ \eps} \|\sum_i\delta_i D_i f_i\|_X^{p}\,.
\end{equation}
In particular,
\begin{equation}
\label{Delta-mf}
\bE_\eps \| \Delta^{1/2}  f\|_X^p \le C(p,q)   \bE_\delta\bE_{ \eps} \|\sum_i\delta_i D_i f\|_X^{p}\,.
\end{equation}
\end{thm}

Inequality \eqref{Delta-m}   is just a little bit more general than results of Francoise Lust-Piquard in \cite{FLP}, \cite{ELP}.

\begin{rem}
\label{Xfinitecotype}
We believe that theorem above holds for all $X$ having finite co-type. But at this moment we cannot prove this.
\end{rem}

\subsection{Non-commutative random variables}
\label{non-c}
Let
$$
Q=\begin{bmatrix}
0 & 1\\
1 & 0
\end{bmatrix},\,\, P = \begin{bmatrix}
0 & i\\
-i & 0
\end{bmatrix},\,\, U= i QP,
$$
They have anti-commutative relationship
\begin{equation}
\label{antiC}
QP=-PQ\,.
\end{equation}
Let $Q_j= I\otimes \dots Q\otimes I\dots\otimes I $, $P_j= I\otimes \dots P\otimes I\dots\otimes I $, on $j$-th place.
These are independent non-commutative random variables in the sense of $\text{tr}$ = sum of diagonal elements divided by $2^n$.

Put $Q_A = \Pi_{i\in A} Q_i$, $P_A = \Pi_{i\in A} P_i$

Now one considers algebra generated by $Q_j, P_j$ (this is algebra of all matrices $\mathcal{M}_{2^n}$). 
We have a projection $\mathcal{P}$ from multi-linear polynomials in $P_j, Q_j$ (notice $P^2=I, Q^2=I$) that kills everything except terms having only $Q's$.

Small (really easy) algebra shows that $\mathcal{P}$ can be written as $\rho Diag \rho^*$, where $\rho$ is a unitary operator, and $Diag$, is an operator on matrices that just kills all matrix elements except the diagonal. This $Diag$ is obviously the contraction on Schatten-von Neumann class $S_p$ for any $p\in [1, \infty]$ (obvious for Hilbert--Schmidt, $p=2$,  class and for bounded operators--so interpolation does that).

Operator $\mathcal{R}(\theta)$ is an automorphism of algebra $\mathcal{M}_{2^n}$ preserving all $S_p$ norms.

$$
\mathcal{R}(\theta)Q_A = \Pi_{j\in A} ( Q_j\cos\theta+ P_j\sin\theta)\,,\,\, \mathcal{R}(\theta)P_A = \Pi_{j\in A} ( P_j\cos\theta - Q_j\sin\theta)\,.
$$
One can easily check that the action of $\mathcal{R}(\theta)$ is $ R(\theta)^* T R(\theta)$ where $R(\theta)$ is a unitary matrix which is $n$-fold tensor product of
$$
\rho_\theta=\begin{bmatrix}
1 & 0\\
0 & e^{i\theta}
\end{bmatrix}
$$

For any $f=\sum_{A\subset [n]} \hat f(A) \eps^A$, the reasoning of \cite{ELP} dictates to assign on non-commutative object, a matrix from $\mathcal{M}_{2^n}$ given by
 $$
 T_f =\sum_{A\subset [n]} \hat f(A) Q_A\,.
 $$
 Such matrices form commutative sub-algebra $M_{2^n}\subset \mathcal{M}_{2^n}$. Operators $\partial_j, D_j$ can be considered on $M_{2^n}$, acting in a canonical way. For example,
 $$
\partial_i Q_A =\begin{cases} Q_{A\setminus i},\,\,\text{if}\,\, i\in A;
\\
0, \, \,\text{if} \,\, i\notin A\,.\end{cases} 
 $$
 
 There is a projection $\mathcal{P}: \mathcal{M}_{2^n} \to M_{2^n}$, which is a contraction in all norms $S_p, 1\le p\le \infty$.

Semigroup $\mathcal{R}(\theta)$ can be written down (see \cite{ELP}) as
\begin{equation}
\label{semi}
\mathcal{R}(\theta) T = e^{\theta \mathcal{D}} (T),\quad \text{where} \,\, \mathcal{D}(T) =\sum_i P_i \partial_i (T), \,  \forall T\in M_{2^n}\,.
\end{equation}

\bigskip

\subsection{Non-commutative formula}
\label{formu}

On Hamming cube we have partial derivatives, acting in the usual way:  $\partial_i \eps^A  = \eps^{A\setminus i}$ if $i\in A$, and 
 $\partial_i \eps^A =0$ otherwise. The discrete derivative $D_i$ is $D_i=\eps_i \partial_i$. Operators $D_i$ are self-adjoint and cube Laplacian is 
 $$
 \Delta = \sum_{i=1}^n D_i\,.
 $$
 
 The inequality that we wish to prove is the following one:
 \begin{equation}
 \label{epi}
 \|\sum_i D_i \Delta^{-1/2} f_i\|_p \le C_p \| \big(\sum |D_i f_i|^2\big)^{1/2}\|_p\,\quad p\ge 2\,.
 \end{equation}
 If all $f$ are the same,  this inequality transforms into 
 \begin{equation}
 \label{flp}
 \|\Delta^{1/2} f\|_p \le C_p \|\nabla f\|_{L^p(\ell^2)},
 \end{equation}
 and the latter was proved in \cite{ELP}.

 \begin{rem}
 Obviously one can write the LHS of \eqref{epi} as $ \|\sum_i D_i \Delta^{-1/2} D_if_i\|_p$ and denote $F_i=D_i f_i$.  Now it becomes tempting to ``generalize" \eqref{epi} to have
  \begin{equation}
 \label{epFi}
 \|\sum_i D_i \Delta^{-1/2} F_i\|_p \le C_p \| \big(\sum |F_i|^2\big)^{1/2}\|_p\,\quad p\ge 2\,.
 \end{equation}
 But this inequality is wrong as if it were correct it would prove just by duality $L^p$-boundedness of  Riesz transform $\vec R:= (D_1 \Delta^{-1/2}, \dots, D_n \Delta^{-1/2})$ on cube for $1<p<2$, which is false, see \cite{ELP}. By the boundedness of the Riesz transform we understand the following inequality for function on discrete cube
 \begin{equation}
 \label{Rp}
 \|\vec R g\|_{L^p (\ell^2)} \le C \|g\|_p\,.
 \end{equation}
 For $1<p<2$ it is proved in \cite{ELP} to be false.
 \end{rem}

 The proof of \eqref{flp}: $\|\Delta^{1/2} f\|_p \le C_p \|\nabla f\|_{L^p(\ell^2)}$  in  \cite{ELP}  uses the formula:
  \begin{equation}
 \label{elpF}
 \Delta^{1/2} f =\mathcal{P}\Big(\int_{-\pi/2}^{\pi/2} \frac{\text{sgn}\,(\theta)}{t(\theta)} e^{\theta \mathcal{D}} \mathcal{D}(T_f)\Big)\,.
 \end{equation}
 Here $t(\theta)$ is given by formula
 $$
 t(\theta)=(-\log\cos\theta)^{1/2}\,.
 $$
 
 This formula seems to be not well-suited to prove \eqref{epi}. 
 
  We need another formula.  First of all
 here is Pisier's lemma Pisier, \cite{P}:
\begin{lem}
\label{le24P}
The odd function $ \frac{\text{sgn}\,(\theta)}{t(\theta)} =\frac{\text{sgn}\,(\theta)}{(-\log\cos\theta)^{1/2}}$ on $[-\pi/2, \pi/2]$  is such that a) $\phi(\theta)- \cot (\theta/2)$ is bounded and 
$$
b) \,\, \forall m\ge 0, \,\, \int_{-\pi/2}^{\pi/2} \cos^m\theta \sin\theta \, \frac{\text{sgn}\,(\theta)}{t(\theta)} \,d\theta = c \frac1{\sqrt{m+1}}\,.
$$
\end{lem}

 The following claim is closely related to Lemma 5.2 of \cite{ELP}. Namely, notice that
 \begin{equation}
 \label{qa}
 D_j \Delta^{-1/2} Q_A = \mathcal{P}\Big(\int_{-\pi/2}^{\pi/2} \frac{\text{sgn}\,(\theta)}{t(\theta)} e^{\theta \mathcal{D}} P_j \partial_j Q_A\Big)\,.
 \end{equation}
 In fact, if $j\notin A$ both sides are zero. Now let $j\in A$. Then 
$$
e^{\theta \mathcal{D}} P_j \partial_j Q_A \!\!  =\!\!e^{\theta \mathcal{D}} P_j Q_{A\setminus j} = \prod_{s\in A, s<j} \!\! (\cos \theta Q_s +\sin\theta P_s)
(\sin\theta \,Q_j -\cos\theta P_j) \!\!  \prod_{s\in A, s>j} \!\! (\cos \theta Q_s +\sin\theta P_s)
$$
Thus
$$
\mathcal{P} \big(e^{\theta \mathcal{D}} P_j \partial_j Q_A\big) = \cos\theta^{|A|-1} \sin\theta \, Q_A
$$
So after integration in $\theta$ (see  Lemma \ref{le24P}) we get
$$
 \mathcal{P}\Big(\int_{-\pi/2}^{\pi/2} \frac{\text{sgn}\,(\theta)}{t(\theta)} e^{\theta \mathcal{D}} P_j \partial_j Q_A\Big) = \frac1{|A|^{1/2}} Q_A =  D_j \Delta^{-1/2} Q_A,
 $$
 as $D_j$ does not change $Q_A$ if $j\in A$.

Formula \eqref{qa} gives us (for any $T=T_f \in M_{2^n}$)
\begin{equation}
\label{pj}
D_j\Delta^{-1/2} (\eps_j \partial_j f)=D_j \Delta^{-1/2} f = \mathcal{P}\Big(\int_{-\pi/2}^{\pi/2} \frac{\text{sgn}\,(\theta)}{t(\theta)} e^{\theta \mathcal{D}} P_j \partial_j T \Big)\,.
 \end{equation}
 
 \begin{rem}
 It is very tempting to ``generalize" this formula and to  write
 $$
D_j\Delta^{-1/2} (\eps_j F)= \mathcal{P}\Big(\int_{-\pi/2}^{\pi/2} \frac{\text{sgn}\,(\theta)}{t(\theta)} e^{\theta \mathcal{D}} P_j T_F\Big)\,.
$$
This is wrong. The difference is that $\eps_j \partial_j f$ is an odd function in $\eps_j$ variable, while $\eps_j F$ does not have this property  for generic $F$.
\end{rem}
\begin{rem}
Adding \eqref{pj} in $j$ we, of course, get \eqref{elpF}.
\end{rem}

\subsection{Commutative end-point inequality by non-commutative formula}
\label{commu}

Having proved formula \eqref{pj} we can finish the proof of  end-point inequality \eqref{epi}.
Let $f_1, \dots, f_n$, we consider $T_1= T_{f_1}, \dots, T_n =T_{f_n}$. Then
$$
\sum_jD_j \Delta^{-1/2} f_j = \mathcal{P}\Big(\int_{-\pi/2}^{\pi/2} \frac{\text{sgn}\,(\theta)}{t(\theta)} e^{\theta \mathcal{D}} \sum_jP_j \partial_j T_j \Big)\,.
$$
Now using the facts that $\mathcal{P}$ is a contraction in $S_p$, that our singular integral  is bounded in $S_p$, $1<p<\infty$, and that semigroup $e^{\theta\mathcal{D}}$ is bounded in $S_p$ we get that
$$
\|\sum_jD_j \Delta^{-1/2} f_j \|_p \le C\,p \,\|\sum_jP_j \partial_j T_j\|_p\,.
$$
Because of anti-commutative relation \eqref{antiC} we notice that 
$$
Q_k \big(\sum_jP_j \partial_j T_j \big) Q_k= \sum_j\eps_jP_j \partial_j T_j,
$$
where all $\eps_i=1$ except for $\eps_k=-1$. Thus, doing this repeatedly, we can see that 
$$
\|\sum_jP_j \partial_j T_j\|_p = \|\sum_j\eps_jP_j \partial_j T_j\|_p
$$
 with arbitrary signs.

\medskip

Using non-commutative Khintchine inequality of Lust-Piquard we get for $p\ge 2$:
$$
\|\sum_jD_j \Delta^{-1/2} f_j \|_p \!\! \le\!\!  C\,p^{3/2} \max \big( \| (\sum_i (\partial_iT_i)^* (\partial_iT_i)\big)^{1/2}\|_p,   \| \big(\sum_i P_i (\partial_iT_i)(\partial_iT_i)^* P_i\big)^{1/2}\|_p\big)\,.
$$
The second term is equal to the first one, because $\partial_i T_i$ does not have $Q_i$ in it, so $P_i$ can be carried  through to make $P_i^2=I$.

So we get 
$$
\|\sum_jD_j \Delta^{-1/2} f_j \|_p \le C\,p^{3/2} \, \| \big(\sum_i |\partial_i f_i|^2\big)^{1/2}\|_p = C\,p^{3/2} \, \| \big(\sum_i |D_i f_i|^2\big)^{1/2}\|_p \,.
$$
This is inequality \eqref{epi} with constant $C_p\asymp p^{3/2}$. 

\begin{rem}
Actually constant's growth  should be $p$ here.  
\end{rem}

\subsection{Transference argument or the Hilbert transform trick}

We need one more step called ``transference argument" or ``the Hilbert transform trick". We need ``to notice" that for a vector-valued function $F(\theta)$ on $[-\pi, \pi]$ with values in functions on  Hamming cube ($F(\theta)=\mathcal{P}(\mathcal{R}(\theta)\sum_{k=1}^n \eps_k \partial_{\delta_k} f )$) the following  four things happen
\begin{equation}
\label{Htr}
a)\,\,\Big\|\int_{-\pi}^\pi  \frac{\text{sgn}\,\theta}{t(\theta)} F(\theta-\psi) \, d\theta\Big\|_{L^p([-\pi, \pi], L^p(\Om_n))} \le C_p \|F\|_{L^p([-\pi, \pi], L^p(\Om_n))}\,,
\end{equation}

\begin{equation}
\label{proj}
b)\,\, \|\mathcal{P}\|_{ L^p(\Om_n)} \le 1 \,,
 \end{equation}
 
 \begin{equation}
\label{equi1}
c)\,\,  \|\mathcal{R}(\theta) G\|_{L^p(\Om_n)}= \|G\|_{L^p(\Om_n)}\,,
 \end{equation}

 \begin{equation}
 \label{equi2}
d)\,\, \Big\|\int_{-\pi}^\pi  \frac{\text{sgn}\,\theta}{t(\theta)} \mathcal{R}(\theta) G \, d\theta\Big\|_{L^p( \Om_n))}= \Big\|\int_{-\pi}^\pi  \frac{\text{sgn}\,\theta}{t(\theta)} \mathcal{R}(\theta-\psi)G \, d\theta\Big\|_{L^p([-\pi, \pi]\times \Om_n)} 
 \end{equation}

 \subsection{Why all this is true?}
 
 $\mathcal{R}(\theta)$ can be written down as 
$$
\mathcal{R}(\theta)T = A_\theta^* T A_\theta,
$$
where $A_\theta$ is a unitary matrix:
$$
A_\theta = R_\theta\otimes\dots \otimes R_\theta,
$$
where
$$
R_\theta := \begin{bmatrix}
1 & 0\\
0 & e^{i\theta}
\end{bmatrix}
$$

\subsection{Why in non-commutative case  the projection $\mathcal{P}$ is the contraction in Schatten--von Neumann norms?}

Consider  matrix 
$$
r:=\frac1{\sqrt{2}}\begin{bmatrix}
1 & 1\\
-1 & 1
\end{bmatrix}
$$ and unitary matrix
$\rho= r\otimes \dots \otimes r$.
Let $\nu$ acts on $\mathcal{M}_{2^n}$  by $\nu(T)=  \rho T\rho^*$.
Then it is easy to calculate that
$$
\nu^{-1} (U_j) = Q_j,\,\, \nu^{-1} (Q_j) = -  U_j = -i P_jQ_j,\,\, \nu^{-1}(P_j) = P_j\,.
$$
Now $\mathcal{P}\nu^{-1}$ kills $Q_j$ and $P_j$ and maps $U_j$ to $Q_j$.
Hence, $\nu \mathcal{P} \nu^{-1}$ maps $U_j$ to $U_j$ and kills others. This shows that $\nu \mathcal{P} \nu^{-1}$ is also a projection (that it is  some projection was clear from the start): $\nu \mathcal{P} \nu^{-1}$ is projection $Diag$ on diagonal. As such, it is a contraction. So $\mathcal{P}$ is a contraction in all $\gamma_p$ norms, $1\le p\le \infty$.
We are done.

\bigskip

\subsection{From scalar valued functions to $L^p(\ell^2)$, $L^p(\ell^p)$ valued functions}
\label{vector}

Recall that on matrices we have Schatten-von Neumann norms:
$$
\|A\|_{\sigma_p} := \Big(\text{Tr}\Big[ (A^*A)^{p/2}\Big]\Big)^{1/p}, \quad 1\le p\le \infty\,.
$$

Let now vector valued $F$ replaces scalar valued $f:\{-1,1\}^n\to  \bR$, namely, $F=\{f^r\}_{r=1}^R$, and each 
$f^r:\{-1,1\}^n\to  \bR$. We consider the mixed metric
$$
\|F\|_{L^p(X)} = \Big\| | \{f^r\}_{r=1}^R|_X\Big\|_{L^p(\{-1, 1\}^n}\,,
$$
in what follows $X= \ell^2_R$ or $X=\ell^p_R$.

Above we have the correspondence (homomorphism) from functions $f:\{-1,1\}^n\to  \bR$ to the class  $2^n\times 2^n$ matrices $M_{2^n}$  given by
$$
f\to T_f
$$
described above. 

Now  let $X=\ell^2_R$ and the correspondence between vector functions $F\in L^p(X)$ and matrices will be
given
$$
\vec F =\{f^r\}_{r=1}^R\to C_F:=\begin{bmatrix} T_{f^1}\\ \cdot\\ \cdot\\ \cdot\\T_{f^R}\end{bmatrix}\,.
$$

Notice that  before we had the isometry
\begin{equation}
\label{sc-p}
\|f\|_{L^p} = \|T_f\|_{\sigma_p},
\end{equation}
Now we also have the same isometry:
\begin{equation}
\label{v-p2}
\|\vec F\|_{L^p(\ell^2_R)} = \|C_F\|_{\sigma_p}\,.
\end{equation}
In fact,
$$
\|C_F\|_{\sigma_p} =\Big( \text{Tr}\Big[ (C_F^*C_F)^{p/2}\Big]\Big)^{1/p} =  \Big( \text{Tr}\Big[ (\sum_{r=1}^R T_{f^r}^*T_{f^r})^{p/2}\Big]\Big)^{1/p}=
$$
$$
\|(\sum_{r=1}^R |f^r|^2)^{1/2}\|_{L^p}=  \|\vec F\|_{L^p(\ell^2_R)},
$$
where the penultimate equation follows from \eqref{sc-p} and from the fact that $f\to T_f$ is algebras homomorphism.

\medskip

One more simple observation follows. Now let us consider $X=\ell^p_R$. The correspondence between vector functions $\vec F\in L^p(X)$ and matrices will be
given
$$
\vec F =\{f^r\}_{r=1}^R\to D_F:=\begin{bmatrix} T_{f^1}, &0 \dots &0 \\ \cdot  &T_{f^2} \dots &0\\\cdot\\ 0\dots  &0 \dots &T_{f^R}\end{bmatrix}\,.
$$
Then, similarly to the above
\begin{equation}
\label{v-pp}
\|\vec F\|_{L^p(\ell^p_R)} = \|D_F\|_{\sigma_p}\,.
\end{equation}
Now, in principle,  we would be able to repeat verbatim the proof for scalar functions. We may hope to use non commutative Khintchine theorem for $\sum_{i=1}^n\eps_i C_{F_i}$ and $\sum_{i=1}^n\eps_i D_{F_i}$ along with \eqref{v-p2}, \eqref{v-pp}. However, the non commutative Khintchine theorem for $\sum_{i=1}^n\eps_i C_{F_i}$ will meet a difficulty.

\bigskip

Let $\Big|\cdot\Big|_2$ denote the norm in $\ell^2_R$.
We choose a more direct way to prove
\begin{equation}
\label{ell-2p}
\bE_\eps \Big| \sum_{i=1}^n \Delta^{-1/2} D_i \vec F_i\Big|_2^p \le C(p) \bE_{\delta, \eps} \Big| \sum_{i=1}^n \delta_i D_i \vec F_i\Big|_2^p\,.
\end{equation}
This follows immediately from Khintchine inequality and the scalar version that we just proved. In fact, as before $\vec F_i= \{f_i^r\}_{r=1}^R$. Let $\{\eps_r'\}_{i=1}^R$ be standard Rademacher variables independent of $\{\eps_i\}_{i=1}^n$ and $\{\delta_i\}_{i=1}^n$.
For each fixed  $\{\eps_r'\}_{i=1}^R$ we just write the scalar inequality
$$
\bE_\eps | \sum_{i=1}^n \Delta^{-1/2} D_i (\sum_{r=1}^R \eps_r' f_i^r)|^p \le C(p) \bE_{\delta, \eps} |  \sum_{i=1}^n \delta_i D_i (\sum_{r=1}^R \eps_r' f_i^r)|^p \,.
$$
Change the order of summation and apply $\bE_{\eps'}$:
$$
\bE_\eps \bE_{\eps'}| \sum_{r=1}^R \eps_r'  \sum_{i=1}^n \Delta^{-1/2} D_i f_i^r|^p \le C(p) \bE_{\delta, \eps}  \bE_{\eps'}|\sum_{r=1}^R \eps_r' \sum_{i=1}^n \delta_i D_i f_i^r|^p \,.
$$
Khintchine inequality applied to both parts gives
$$
\bE_\eps \Big( \sum_{r=1}^R  | \sum_{i=1}^n \Delta^{-1/2} D_i f_i^r|^2\Big)^{p/2} \le C'(p)\bE_{\delta, \eps}  \Big(\sum_{r=1}^R | \sum_{i=1}^n \delta_i D_i f_i^r|^2\Big)^{p/2} \,.
$$
This is \eqref{ell-2p}.

\bigskip

Now let $\Big|\cdot\Big|_p$ denote the norm in $\ell^p_R$.
The next inequality does not need to be proved, it just follows from  the summation of $R$ scalar inequalities.
\begin{equation}
\label{ell-pp}
\bE_\eps \Big| \sum_{i=1}^n \Delta^{-1/2} D_i \vec F_i\Big|_p^p \le C(p) \bE_{\delta, \eps} \Big| \sum_{i=1}^n \delta_i D_i \vec F_i\Big|_p^p\,.
\end{equation}

Now one can interpolate between \eqref{ell-2p} with norm $L^p(\ell^2_R)$ and \eqref{ell-pp} norm $L^p(\ell^p_R)$ and get the inequality
\begin{equation}
\label{ell-qp}
\bE_\eps \Big| \sum_{i=1}^n \Delta^{-1/2} D_i \vec F_i\Big|_q^p \le C(p) \bE_{\delta, \eps} \Big| \sum_{i=1}^n \delta_i D_i \vec F_i\Big|_q^p, \quad, 2\le q\le p\,.
\end{equation}
This means that we finally proved Theorem \ref{Delta12}.

\section{Counterexamples}
\label{cex}

\subsection{Counterexample 1: Talagrand's counterexample}
\label{cexT}

The inequality 
$$
||F-\bE F||_{L^p(X)}^p\lesssim  \bE_\delta||\sum_i \delta_i D_iF ||_{L^p(X)}^p
$$
 does not hold for
general Banach space valued function $F$ (without finite co-type requirement) when $p<\infty$ with constant independent of $n$.

To construct vector-valued counterexample one starts with the scalar case.  In the scalar case, we know that certain
inequality, namely \eqref{failTsc}, fails for bounded scalar valued functions, with counterexample given by this concrete function:
$$
f(\eps)=\log^+\frac{(\sum_i 
(1-\eps_i)/2)}{\sqrt{n}}=\log^+\frac{\text{dist}(\eps, 1)}{\sqrt{n}}\,.
$$

The relevant scalar inequality that fails is
\begin{equation}
\label{failTsc}
    ||f-\bE f||_{L^\infty(\eps)} \lesssim \Big(\bE_\delta \|\sum_i \delta_i D_i f\|_{L^\infty(\eps)}^p\Big)^{1/p}=
    \Big\| ||\sum_i \delta_i D_if||_{L^\infty(\eps)} \Big\|_{L^p(\delta)}
\end{equation}
We will see how from failing \eqref{failTsc} to determine the Banach space $X$ and a $X$ valued function $f$ on Hamming cube in such a way that the following will fail too:
\begin{equation}
\label{failTtoo}
   \Big(\bE_\eps ||f-\bE f||_X^p\Big)^{1/p} \lesssim \Big(\bE_{\delta, \eps} \|\sum_i \delta_i D_i f\|_{X}^p\Big)^{1/p}=
    \Big\| ||\sum_i \delta_i D_if||_{X}\Big\|_{L^p(\delta, \eps)}
\end{equation}
For scaler $f$ above 
$$
A_1=\bE_x \log^+\frac{\text{dist}(x, 1)}{\sqrt{n}} \ge c \log n,
$$
because the distance is at least $n/3$ for a fixed portion of $x$'s.
On the other hand,  if $\text{dist}(\eps, 1)<\sqrt{n}$, then $f(\eps)=0$. The portion of such $\eps$ is exponetially small but positive. Therefore,
\begin{equation}
\label{lhsTsc}
\|f(\eps) - \bE f\|_{L^\infty(\eps)} \ge c\log n\,.
\end{equation}

Now let us look at the estimate of $(\bE_{\eps,\delta}\|\sum_i \delta D_i f\|_X^p)^{1/p}$. According to (6.2)  on page 313 of \cite{T93}
\begin{equation}
\label{ab}
\forall \delta, \eps \quad |\sum_i \delta_i D_{i} f(\eps)| \le \frac{4}{\sqrt{n}} |\sum_i \delta_i | + 4\,.
\end{equation}
Hence,
\begin{equation}
\label{rhsTsc}
\Big \| ||\sum_i \delta_i D_if||_{L^\infty(\eps)} \Big\|_{L^p(\delta)}^p  \le \frac{C_1}{n^{p/2}}\bE_\delta  |\sum_i \delta_i | ^p +C_2\le C_3\,.
 \end{equation}
 
Comparing \eqref{lhsTsc} and \eqref{rhsTsc}, we see that \eqref{failTsc} really fails.

Now we are going to show Talagrand's example by considering $L_\eps^p(L_\eta^\infty(\{-1, 1\}^n)$ and vector valued function (or just a function of two variables, $\eps, \eta$ with mixed norm. As the recipe says in the vector  valued case with $p<\infty,$ we choose 
$X=L^\infty(\{-1,1\}^n)$ and consider $(\eps,\eta) \to f(\eps\eta)$. By symmetry, we can 
exchange the roles of $\eps$ and $\eta$ in the inequality. This way we embed 
the scalar example for $p=\infty$ into the vector case for $p<\infty$, namely into $L^p(\ell^\infty)$.

So, consider 
$$
F(\eps;\eta):=f(\eps\eta)=\log^+\frac{(\sum_i 
(1-\eps_i\eta_i)/2)}{\sqrt{n}}\,.
$$
$$
A_\eta=\bE_x \log^+\frac{\text{dist}(x, \eta)}{\sqrt{n}} \ge c \log n,
$$
By group invariance 
$$
\| F(\eps; \eta) -\bE_\eps F(\eps; \eta)\|_{L^\infty(\eta)} =\| F(\eps; \eta) -A_\eta\|_{L^\infty(\eta)}   \ge c\log n\,.
$$
In fact, the left hand side is independent of $\eps$.
Hence, we have a full analog of \eqref{lhsTsc}:
\begin{equation}
\label{lhsTX}
\Big\| || F - A_\eta||_{L^\infty_\eta}\Big\|_{L^p(\eps)} \ge c\log n\,.
\end{equation}

On the other hand, the same group invariance shows us that \eqref{ab} implies
\begin{equation}
\label{abX}
\forall \delta, \eps, \eta \quad |\sum_i \delta_{i}(D_{\eps_i} F)(\eps; \eta)| \le \frac{4}{\sqrt{n}} |\sum_i \delta_i | + 4\,.
\end{equation}
From here we get the analog of \eqref{rhsTsc}:
\begin{equation}
\label{rhsTX}
\Big \| ||\sum_i \delta_i D_if||_{L^\infty(\eta)} \Big\|_{L^p(\delta,\eps)}^p  \le \frac{C_1}{n^{p/2}}\bE_\delta  |\sum_i \delta_i | ^p +C_2\le C_3\,.
 \end{equation}

Comparing \eqref{lhsTX} and \eqref{rhsTX} we  come to Talagrand's counterexample:
$$
||F-\bE F||_{L^p(X)}^p\lesssim  \bE_\delta||\sum_i \delta_i D_iF ||_{L^p(X)}^p
$$
 does not hold in $L^p(X)=L^p(\{-1, 1\}^n, L^\infty(\{-1,1\}^n))$ with constant independent of $n$ and even with constant growing slower than $c\log n$.

\subsection{Counterexample 2: Banach space $X$ can be very good but there is no Riesz transform estimate from above for $X$-valued functions on discrete cube}
\label{cexRieszXfrAbove}

We show below that the Riesz transform inequality 
\begin{equation}
\label{RX1}
||\sum_i \delta_i D_i \Delta^{-1/2} f||_p \lesssim ||f||_p
\end{equation}
does not hold in $L^p(X)$ for  Banach spaces $X$ with  co-type $2$ when $p\ge 2$.

To construct vector-valued counterexample one starts with the scalar case.  The idea follows Talagrand's counterexample in the previous section, the counterexample below was given by Ramon Van Handel.  In the scalar case with $\alpha<1$, we know 
the inequality fails for $p<1/\al$ 
$$
||\sum_i \delta_i D_i \Delta^{-\al} f||_{L^p(X)} \lesssim ||f||_{L^p(X)}
$$
with counterexample given by a special scalar function $f$. See Lemma 5.5 in \cite{ELP}. So, in 
the vector case with $p>2$, we choose $X=\ell^{p'}$ with $p'<2$, and consider 
the function $(\eps,\eta) \to f(\eps\eta)$. By symmetry, we can exchange the 
roles of $\eps$ and $\eta$ in the inequality. This way we embedded the scalar 
counterexample for $p'<2$ into the vector case for $p>2$.

\begin{rem}
\label{pb2}
The striking thing here is not that the Riesz transform inequality is false even if coefficients are in a very good Banach space $X$, it will be $L^s, 1<s<2$, so in particular a $UMD$ space, such things happen, see Lamberton's example on page 283 of \cite{FLP}, but more striking is that this happens for $X$-valued $L^p$ with $p\ge2$. For discrete cube the case of $L^p(X)$ turns out to be mysterious even for $p\ge 2$.
\end{rem}

\begin{rem}
\label{webelieve}
We believe that this effect of ``good spaces/bad estimate of Riesz transform on discrete cube'' appears because \eqref{RX1} is a vector valued estimate of Riesz transform from above.
For the estimate of a vector valued  Riesz transform from below, and even for a more general estimate
\begin{equation}
\label{RXbelow1}
\|\sum_i D_i \Delta^{-1/2} f_i\|_{L^p(X)} \lesssim \Big(\bE_\delta\|\sum_i \delta_i D_i f_i\|_{L^p(X)}^p\Big)^{1/p}
\end{equation}
we believe that it holds for all $1<p<\infty$ and all $X$ of finite co-type. As the reader saw in Section \ref{qua} we managed to prove this only for certain regimes $(p, q)$ for $X=L^q, q<\infty$.
\end{rem}

Along with counterexample to \eqref{RX1} written below we could have written a more general one, namely, that for $\al<1$
\begin{equation}
\label{RX2}
||\sum_i \delta_i D_i \Delta^{-\al} f||_{L^p(X)}  \lesssim ||f||_{L^p(X)} 
\end{equation}
does not hold in $L^p(X)$ if $X$ is chosen to be $L^s, 1< s<1/\al$ and $p\ge s$.

\bigskip

Now the details follow. Lamberton's function that Lust-Piquard  \cite{FLP} uses  as a counterexample in the scalar case for 
$s<2$ is the function $f:\{-1,1\}^n\to \bR$ defined by
$$
   f(\epsilon) := 2^{-n} \prod_{i=1}^n (1+\epsilon_i) = 1_{\epsilon={\bf 1}}.
$$
But by equivariance of everything (derivatives, Laplacian, etc.) under 
action of the group $\{-1,1\}^n$, this counterexample works verbatim for the 
function $f_\eta:\{-1,1\}^n\to \bR$ for any $\eta\in\{-1,1\}^n$ defined by
$$
   f_\eta(\epsilon) := f(\epsilon\eta) = 1_{\epsilon=\eta}.
$$
We now define the vector valued function $g:\{-1,1\}^n\to L^s(\{-1,1\}^n)$,  $s<2$, by setting
$$
   g(\epsilon) = f_.(\epsilon)
$$
(that is, for every $\epsilon$, we think of the map $\eta \to  f_\eta(\epsilon)$ 
as an element of $X:=L^s(\{-1,1\}^n)$. We claim that
\begin{equation}
\label{cannot}
  \bE_\delta || \sum_i \delta_i D_i g(\epsilon)||_{L^p(X)}^p \lesssim ||\Delta^{1/2} g||_{L^p(X)}^p
\end{equation}
cannot hold with dimension-free constant for the above $X$ and $g$, for any 
$p\ge 2$ and $s<2$, providing us with the counterexample we seek with a strikingly good Banach space $X:=L^s(\{-1,1\}^n)$, it is $UMD$ for example.

To see it, we basically just exchange the roles of $\epsilon$ and $\eta$ in 
the above inequalities. After all, clearly,
$
D_{\eps_i} \big[f(\eps\eta)\big] = (D_{i}f)(\eps\eta)
$, and so
$$
   ||\sum_i \delta_i D_i g(\epsilon)||_X
   = [E_\eta \big|\sum_i \delta_i D_i f_\eta(\epsilon)\big|^s]^{1/s}
   = [E_\eta\big|\sum_i \delta_i D_i f_\epsilon(\eta)\big|^s]^{1/s}  =
$$
$$
[E_\eta\big|\sum_i \delta_i D_i f(\eta)\big|^s]^{1/s}, 
$$
where we used the symmetry of $f$ in $\eta,\epsilon$. 
Using this we see that $E_\eta \big|\sum_i \delta_i D_i f_\epsilon(\eta)\big|^s$, $E_\eta\big|Df_\epsilon(\eta)\big|^s$ are independent of $\eps$ and using Khintchine inequality, we find that \eqref{cannot} would imply (as $p/s>1$) the following:
$$
  [E_\eta\big|Df(\eta)\big|^s]^{p/s}
 \asymp \Big(\bE_\eta \bE_\delta \big|\sum_i \delta_i D_i f(\eta)\big|^s\Big)^{p/s}= \Big(\bE_\delta \bE_\eta \big|\sum_i \delta_i D_i f(\eta)\big|^s\Big)^{p/s} \le
$$
$$
 \bE_\delta\Big([E_\eta\big|\sum_i \delta_i D_i f(\eta)\big|^s]^{1/s}\Big)^p=\bE_\delta\bE_\epsilon\Big([E_\eta\big|\sum_i \delta_i D_i f_\epsilon(\eta)\big|^s]^{1/s}\Big)^p=
$$
$$
    \bE_\delta || \sum_i \delta_i D_i g(\epsilon)||_{L^p(X))}^p
   \lesssim ||\Delta^{1/2} g||_{L^p(X)}^p
   = \big(\bE_\eps [E_\eta\big|\Delta^{1/2} f_\epsilon(\eta)\big|^s]^{1/s}\big)^p=
$$
$$
[E_\eta\big|\Delta^{1/2} f(\eta)\big|^s]^{p/s}\,.
$$
This is because the expression $E_\eta\big|(\Delta^{1/2} f_\epsilon(\eta)\big|^s$ is also independent of $\eps$.
But  Lamberton's (unpublished)  scalar counterexample that can be found on p. 283 of  Lust-Piquard's \cite{FLP} with $s\in (1, 2)$ shows  that
$$
   [E_\eta|Df(\eta)|^s]^{1/s}
   \lesssim [E_\eta(\Delta^{1/2} f(\eta))^s]^{1/s}
$$
cannot hold with dimension-free constant. Thus we have the counterexample we want.

\bigskip

\subsection{Riesz transforms from below. No counterexample}
\label{nocex}

When viewed this way,  our counterexample in subsection \ref{cexRieszXfrAbove} and Talagrand's  example in subsection \ref{cexT} really rely 
on the same principle: find a scalar inequality that fails and then embed 
it into the vector inequality by a proper choice of space $X$ and symmetry.
The same principle can be applied to other cases. 
For example, to disprove the following inequality for general Banach space
$$
||\sum_i \delta_i D_i \Delta^{-1} f||_{L^p(X)} \lesssim ||f||_{L^p(X)}
$$
we can show that the scalar inequality is false in $L^1$ and consider the same scheme as above with $p\ge 1$ and $X=L^1(\{-1, 1\}^n)$.
Now let us return to our question: can we  disprove
\begin{equation}
\label{RXbelow3}
    ||\Delta^{1/2-\gamma} f||_{L^p(X)}  \lesssim \Big(\bE_\delta||\sum_i \delta_i D_if||_{L^p(X)}^p\Big)^{1/p}   
\end{equation}
for all Banach spaces of finite co-type when $0<\gamma<1/2, 1\le p<\infty$ ? 
Following the only way we know to generate counterexamples, we should find 
a space $X$ of  scalar functions
$$
 ||\Delta^{1/2-\gamma} f||_{X}  \lesssim \Big(\bE_\delta||\sum_i \delta_i D_if||_{X}^p\Big)^{1/p}   
 $$
 that fails, and then embed it in the vector valued case. But the scalar 
inequality is valid for all $1\le p<\infty$, so the only 
case of failure for the scalar inequality is $L^\infty$ --- and this  space does 
not have finite co-type. So there is no hope of disproving \eqref{RXbelow3} for finite 
co-type by adapting any of the examples we have so far. 
\begin{rem}
And, in fact, we 
 believe that \eqref{RXbelow3} should be valid for 
every $0<\gamma<1/2, 1\le p<\infty,$ and Banach space of finite co-type.
\end{rem}

\section{Open problems}
\label{open}

We left many question concerning singular integrals on discrete cibe unanswered. They concern Banach space valued questions as well as scalar valued functions.

\subsection{Scalar valued functions and singular integrals on discrete cube}

In Section \ref{qua} we obtained the estimate 
$$
\|\sum_i\Delta^{-1/2} D_i f_i\|_p \le Cp^{3/2} \Big(\bE_\delta \|\sum_i \delta_i D_i f_i\|_p^p\Big)^{1/p},\quad 2\le p<\infty\,.
$$
Bourgain's paper \cite{B} indicate that the constant should be of order $p$. How to improve the quantum prove above to get a better constant? Our other approach via Bellman function (it will be presented elsewhere) also seems to give constant of order $p^{3/2}$.

For the more difficult Riesz estimate from above we know that for $1<p<2$ this cannot be true: $\||\nabla f|_{\ell^2}\|_p \lesssim \|\Delta^{1/2}f\|_p$. But maybe the following is true
$$
\||\nabla f |_{\ell^2}\|_p \lesssim \|\Delta^{1/p}f\|_p, \quad 1<p<2\,?
$$
This question seems very enticing, and it has been already asked in \cite{ELP}.

\bigskip

\subsection{Banach space valued functions and singular integrals on discrete cube}

The following inequality should be true for a wide class of Banach spaces $X$, but it is proved above only for some special $X$:
\begin{equation}
\label{RXbelowGamma}
    ||\Delta^{1/2-\gamma} f||_{L^p(X)} \lesssim \Big(\bE_\delta||\sum_i \delta_i D_if||_{L^p(X)}^p  \Big)^{1/p}\,. 
\end{equation}
Is it true for all Banach spaces of finite co-type when $0<\gamma<1/2, 1\le p<\infty$ ? For very particular case of $X=L^q$  and some $(p,q)$ regimes this is easy. But  the regimes $2\le p<q <\infty$ this  and $1<p<2<q$ are left without an answer. See the discussion of the difficulty of constructing a counterexample in Subsection \ref{nocex} above.

We can ask the same questions for the ``ultimate'' Riesz transform estimate from below, namely, for $\gamma=0$.  Consider the estimate
\begin{equation}
\label{RXbelow4}
    ||\Delta^{1/2} f||_{L^p(X)} \lesssim \Big(\bE_\delta||\sum_i \delta_i D_if||_{L^p(X)}^p  \Big)^{1/p}\,. 
\end{equation}
Is it true for all Banach spaces of finite co-type when $1< p<\infty$ ?  For $X=L^q$, $2\le q\le p$ it is Theorem \ref{Delta12} above.

\medskip

The difficulty here is exacerbated by the following observation. Above inequality is of the type the ``Riesz estimate from below'', and, as such, the intuition says that it should be very widely true. On the other hand, we are used that singular integrals with Banach space valued function have to be true at least in $L^p(X)$ if $X$ is $UMD$ and $p\ge 2$. But on discrete cube this is not true anymore. We saw a counterexample above with $X=L^s, 1<s<2$, showing that
$$
 \bE_\delta || \sum_i \delta_i D_i g||_{L^p(X)}^p \lesssim ||\Delta^{1/2} g||_{L^p(X)}^p, \quad 2\le p<\infty,
$$
cannot be true with constant independent of $n$. See  \eqref{cannot}  in Subsection \ref{cexRieszXfrAbove}.
This is  inequality is of the type the ``Riesz estimate from above'', but for scalar functions it is true for $2\le p<\infty$!

\end{document}